\documentclass[a4paper,11pt]{amsart}
\usepackage{amsmath,amsthm,amssymb,latexsym,enumerate,color,hyperref}
\usepackage{graphicx}
\usepackage{esint}
\usepackage{comment}
\usepackage{eucal}

\numberwithin{equation}{section}

\begin{document}

\newtheorem{thm}{Theorem}[section]
\newtheorem{prop}[thm]{Proposition}
\newtheorem{lem}[thm]{Lemma}
\newtheorem{cor}[thm]{Corollary}
\newtheorem{rem}[thm]{Remark}
\newtheorem*{defn}{Definition}

\newcommand{\cb}{\color{blue}}

\newcommand{\DD}{\mathbb{D}}
\newcommand{\NN}{\mathbb{N}}
\newcommand{\ZZ}{\mathbb{Z}}
\newcommand{\QQ}{\mathbb{Q}}
\newcommand{\RR}{\mathbb{R}}
\newcommand{\CC}{\mathbb{C}}
\renewcommand{\SS}{\mathbb{S}}

\renewcommand{\theequation}{\arabic{section}.\arabic{equation}}

\newcommand\ddfrac[2]{\frac{\displaystyle #1}{\displaystyle #2}}

\newcommand{\supp}{\mathop{\mathrm{supp}}}    

\newcommand{\re}{\mathop{\mathrm{Re}}}   
\newcommand{\im}{\mathop{\mathrm{Im}}}   
\newcommand{\dist}{\mathop{\mathrm{dist}}}  
\newcommand{\link}{\mathop{\circ\kern-.35em -}}
\newcommand{\spn}{\mathop{\mathrm{span}}}   
\newcommand{\ind}{\mathop{\mathrm{ind}}}   
\newcommand{\rank}{\mathop{\mathrm{rank}}}   
\newcommand{\Fix}{\mathop{\mathrm{Fix}}}   
\newcommand{\codim}{\mathop{\mathrm{codim}}}   
\newcommand{\conv}{\mathop{\mathrm{conv}}}   
\newcommand{\osc}{\mathop{\mathrm{osc}}}  
\newcommand{\epsi}{\mbox{$\varepsilon$}}
\newcommand{\eps}{\mathchoice{\epsi}{\epsi}
{\mbox{\scriptsize\epsi}}{\mbox{\tiny\epsi}}}
\newcommand{\cl}{\overline}
\newcommand{\pa}{\partial}
\newcommand{\ve}{\varepsilon}
\newcommand{\zi}{\zeta}
\newcommand{\Si}{\Sigma}
\newcommand{\cD}{{\mathcal D}}
\newcommand{\cE}{{\mathcal E}}
\newcommand{\cF}{{\mathcal F}}
\newcommand{\cG}{{\mathcal G}}
\newcommand{\cH}{{\mathcal H}}
\newcommand{\cI}{{\mathcal I}}
\newcommand{\cJ}{{\mathcal J}}
\newcommand{\cK}{{\mathcal K}}
\newcommand{\cL}{{\mathcal L}}
\newcommand{\cN}{{\mathcal N}}
\newcommand{\cR}{{\mathcal R}}
\newcommand{\cS}{{\mathcal S}}
\newcommand{\cT}{{\mathcal T}}
\newcommand{\cU}{{\mathcal U}}
\newcommand{\B}{\bullet}
\newcommand{\ol}{\overline}
\newcommand{\ul}{\underline}
\newcommand{\vp}{\varphi}
\newcommand{\AC}{\mathop{\mathrm{AC}}}   
\newcommand{\Lip}{\mathop{\mathrm{Lip}}}   
\newcommand{\es}{\mathop{\mathrm{esssup}}}   
\newcommand{\les}{\mathop{\mathrm{les}}}   
\newcommand{\nid}{\noindent}
\newcommand{\pzr}{\phi^0_R}
\newcommand{\pir}{\phi^\infty_R}
\newcommand{\psr}{\phi^*_R}
\newcommand{\pow}{\frac{N}{N-1}}
\newcommand{\ncl}{\mathop{\mathrm{nc-lim}}}   
\newcommand{\nvl}{\mathop{\mathrm{nv-lim}}}  
\newcommand{\la}{\lambda}
\newcommand{\La}{\Lambda}    
\newcommand{\de}{\delta}    
\newcommand{\fhi}{\varphi} 
\newcommand{\ga}{\gamma}    
\newcommand{\ka}{\kappa}   

\newcommand{\core}{\heartsuit}
\newcommand{\diam}{\mathrm{diam}}

\newcommand{\lan}{\langle}
\newcommand{\ran}{\rangle}
\newcommand{\tr}{\mathop{\mathrm{tr}}}
\newcommand{\diag}{\mathop{\mathrm{diag}}}
\newcommand{\dv}{\mathop{\mathrm{div}}}

\newcommand{\al}{\alpha}
\newcommand{\be}{\beta}
\newcommand{\Om}{\Omega}
\newcommand{\na}{\nabla}

\newcommand{\cC}{\mathcal{C}}
\newcommand{\cM}{\mathcal{M}}
\newcommand{\nr}{\Vert}
\newcommand{\De}{\Delta}
\newcommand{\cX}{\mathcal{X}}
\newcommand{\cP}{\mathcal{P}}
\newcommand{\om}{\omega}
\newcommand{\si}{\sigma}
\newcommand{\te}{\theta}
\newcommand{\Ga}{\Gamma}


\title[An integral identity and the reverse Serrin problem]{A general integral identity with applications to a reverse Serrin problem}

\author[Magnanini]{Rolando Magnanini}
\address{Dipartimento di Matematica ed Informatica ``U.~Dini'',
Universit\` a di Firenze, viale Morgagni 67/A, 50134 Firenze, Italy.}
    \email{rolando.magnanini@unifi.it}
    \urladdr{https://people.dimai.unifi.it/magnanini/}

\author[Molinarolo]{Riccardo Molinarolo}
\address{Dipartimento di Matematica ed Informatica ``U.~Dini'',
Universit\` a di Firenze, viale Morgagni 67/A, 50134 Firenze, Italy.}
\curraddr{Dipartimento per lo Sviluppo Sostenibile e la Transizione Ecologica (DiSSTE), Universit\`a degli Studi del Piemonte Orientale ``A. Avogadro",
Complesso San Giuseppe -- Piazza Sant'Eusebio 5,
13100 Vercelli (IT)
}
    \email{riccardo.molinarolo@uniupo.it}
    
\author{Giorgio Poggesi}
\address{Department of Mathematics and Statistics, The University of Western Australia, 35 Stirling Highway, Crawley, Perth, WA 6009, Australia}

    \email{giorgio.poggesi@uwa.edu.au}


\begin{abstract}
We prove a new general differential identity and an associated integral identity, which entails a pair of solutions of the Poisson equation with constant source term. This generalizes a formula that the first and third authors previously  proved and used to obtain quantitative estimates of spherical symmetry for the Serrin overdetermined boundary value problem.
As an application,
we prove 
a quantitative symmetry result for the \textit{reverse Serrin problem}, which we introduce for the first time in this paper.  
In passing, we obtain a rigidity result for solutions of the aforementioned Poisson equation subject to a constant Neumann condition. 
\end{abstract}

\keywords{Integral identities, Serrin overdetermined problem, rigidity results, stability, quantitative symmetry}
\subjclass[2010]{Primary 35N25, 35B35, 35M12; Secondary 35A23}

\maketitle

\raggedbottom

\section{Introduction}
Let $\Om$ be a bounded domain in the Euclidean space $\RR^N$, $N\ge 2$, with sufficiently regular boundary $\Ga$. We consider solutions of the Poisson equation: 
\begin{equation}
\label{torsion}
\De u=N \ \mbox{ in } \ \Om.
\end{equation}
In \cite{MP1, MP2}, the first and third author of this paper stated and then proved the following integral identity: 
\begin{equation}
\label{identity-for-serrin}
\int_{\Om} (-u) \left\{ |\na ^2 u|^2- \frac{ (\De u)^2}{N} \right\} dx=
\frac{1}{2}\,\int_\Ga \left( u_\nu^2- R^2\right) (u_\nu-q_\nu)\,dS_x.
\end{equation}
Here, $u$ is a solution of \eqref{torsion} such that
\begin{equation}
\label{dirichlet-homogeneous}
u=0 \ \mbox{ on } \ \Ga
\end{equation}
and $q=q^z$ is  a quadratic polynomial of the form
\begin{equation}
\label{quadratic}
q^z(x)=\frac12\,|x-z|^2 +a\ \mbox{ for } \ x\in\RR^N,
\end{equation}
for some $z\in \RR^N$ and $a\in \RR$.
Moreover, $\nu$ is the exterior unit normal on $\Ga$
and the constant $R$ has the value
\begin{equation}
\label{def-R}
R=\frac{N\,|\Om|}{|\Ga|}.
\end{equation}
Also, notice that  it is easy to see that
$$
|\na ^2 u|^2-\frac{(\De u)^2}{N}=\De P,
$$
when we set:
$$
P=\frac12\,|\na u|^2-u.
$$
This is a standard $P$-function for \eqref{torsion}.
\par
The function $\De P$ is important. It is non-negative, by the Cauchy-Schwarz inequality and vanishes if and only if $u$ is a quadratic polynomial as in \eqref{quadratic}. Thus, it gives a \textit{quadratic-radiality test} for $u$. 
By this important feature, \eqref{identity-for-serrin}  yields the Serrin Symmetry Theorem (see \cite{Se} for the  celebrated proof by the method of moving planes and \cite{We} for the first proof based on integral identities). 
This asserts that, if the solution $u$ of \eqref{torsion} and \eqref{dirichlet-homogeneous} has constant normal derivative $u_\nu$ on $\Ga$, then $\Om$ must be a ball. This is now easily seen from \eqref{identity-for-serrin}, since it turns out that the constant value of $u_\nu$ on $\Ga$ must equal $R$, by the divergence theorem. Thus, \eqref{identity-for-serrin}  gives that $\De P\equiv 0$ (and hence radial symmetry), being as $-u>0$ on $\Om$, by the strong maximum principle. The use of the volume integral in \eqref{identity-for-serrin} to test the radiality of $u$ was first noticed in \cite{PS}.
\par
Besides giving rigidity, \eqref{identity-for-serrin} 
is even more important, because it allows to construct sharp quantitative estimates on how the relevant domain $\Om$ deviates from being a ball (in the Hausdorff topology of domains) in terms of 
some norm of the deviation of $u_\nu$ from a constant.
This theme was inaugurated in \cite{ABR}, by using a quantitative version of the method of moving planes. There, a quantitative symmetry estimate is obtained in terms of a logarithmic profile of a $C^1(\Ga)$-norm of the aforementioned deviation. It is worth noticing that this approach also works for positive solutions of quite general semilinear Poisson equations. That strategy was improved in \cite{CMV} thus obtaining  a quantitative symmetry estimate in terms of a polynomial profile of the Lipschitz seminorm of $u_\nu$. (See also \cite{CM, CMS1, CMS2} for related research and \cite{CDPPV, DPTV1, DPTV2} for recent generalizations to the fractional setting.)
\par
An approach based on a quantitative version of the arguments contained in \cite{We} was initiated in \cite{BNST}. The method only works for the Poisson equation \eqref{torsion}. However, the measure of the relevant deviation of $u_\nu$ from a constant is greatly relaxed by using a Lebesgue norm on $\Ga$.  (We refer the interested reader to \cite{Ma, MP5, PogTesi} for a survey on the several stability results available in the literature; see also \cite{CPY, DPVSerrin} for related results in different settings.)
\par
Identity \eqref{identity-for-serrin} was used by
the first and third author to obtain   optimal (for low dimension) and allegedly\footnote{Here, we mean that we conjecture that the bounds obtained in \cite{MP5} are indeed optimal.}optimal (for large dimension) quantitative bounds. This was the theme of a series of papers  (\cite{MP2, MP3, MP5}) culminating with the proof
of the inequality
\begin{equation}
\label{stability-classical-serrin}
\rho_e-\rho_i\le c\,\psi(\nr u_\nu-R\nr_{2,\Ga}).
\end{equation}
(See also \cite{MP6, Po, PPR} for recent extensions to mixed boundary value problems.) In \eqref{stability-classical-serrin}, we mean that
\begin{equation}
\label{def-rhos}
\rho_e=\max_{x\in\Ga}|x-z|, \quad \rho_i=\min_{x\in\Ga}|x-z|,
\end{equation}
for some point $z\in\Om$. Hence, $\rho_i$ and $\rho_e$ are the radii of the largest ball contained in $\Om$ and the smallest ball containing $\Om$, both centered at $z$.\footnote{ In \cite{MP5}, it is also shown that, besides $\rho_e-\rho_i$, the $L^2$-deviation of the Gauss map of $\Ga$ from that of a sphere can be estimated by the same profile $\psi$.}
\par
The profile $\psi:(0,\infty)\to(0,\infty)$ changes with the dimension. 
In fact, by \cite{MP3} (and the improvement provided by \cite{MP5} in the case $N=3$), we have that
\begin{equation}
\label{eq:old-profile-psi}
\psi(t)=\begin{cases}
           t
            & \text{if } N=2 ,
            \\
            t \max \left[\log \left(1/t\right), 1 \right]  
            & \text{if } N=3 ,
            \\
            t^{2/(N-1)}          
            & \text{if } N \geq 4,
        \end{cases}
\end{equation}
for $t>0$.
The constant $c$ in \eqref{stability-classical-serrin} only depends on the dimension $N$, the diameter $d_\Om$ of $\Om$, and the radii $r_i$ and $r_e$ of the uniform interior and exterior sphere conditions of $\Om$,  as defined in Section \ref{sec:preliminary}. (We recall that 
those conditions
are equivalent to the $C^{1,1}$ regularity of $\Ga$ as shown in \cite[Corollary 3.14]{ABMMZ}.) In \cite{MP5} improvements of the profile are also provided for $N\ge 4$. In particular, if $\Ga$ is of class $C^{2,\al}$, then \cite[Theorem 4.4]{MP5} ensures that \eqref{stability-classical-serrin} holds true with 
\begin{equation}\label{eq:improvementMinE}
	\psi(t)= t^{4/(N+1)} \quad \text{ for }  N\ge 4,
\end{equation}
provided $c$ depends on the $C^{2,\al}$-regularity of $\Ga$, instead of its $C^{1,1}$-regularity.

\medskip

In this paper, we derive an integral identity, which generalizes \eqref{identity-for-serrin}. 
It simply involves any two solutions $u$ and $v$ of class $C^2(\ol{\Om})$ of \eqref{torsion}, and reads as: 
\begin{multline}
\label{fundamental-identity}
\int_\Om (\ol{u}-u)\,\De P\,dx +  \int_\Om \lan (I - \na^2 v) \na u, \na u\ran \,dx =\\
\int_\Ga (\ol{u}-u)\,\lan\na^2 u\, \na u,\nu\ran\,dS_x +
\frac12\,\int_\Ga |\na u|^2 (u_\nu+ v_\nu)\,dS_x+ \\
-\int_\Ga \lan \na v,\na u \ran \,u_\nu\,dS_x 
-\int_\Ga (\ol{u}-u)\,u_\nu\,dS_x
+
N \int_\Ga (\ol{u}-u) (u_\nu - v_\nu) \,dS_x.
\end{multline}
The proof is a simple application of the divergence theorem to an appropriate differential identity (see Theorem \ref{th:fundamental-identity} for details).
\par
In \eqref{fundamental-identity}, we mean that
$$
\ol{u}=\max_{\ol{\Om}} u=\max_\Ga u,
$$
so that $\ol{u}-u>0$ in $\Om$, by the strong maximum principle.
Of course, if we require that $u$ satisfies \eqref{dirichlet-homogeneous} or, more generally, 
\begin{equation}
\label{dirichlet}
u=u_0 \ \mbox{ on } \ \Ga,
\end{equation}
for some $u_0\in\RR$, \eqref{fundamental-identity} will simply return \eqref{identity-for-serrin} (in the latter case with $-u$ replaced by $\ol{u}-u=u_0-u$).
\par
The significant generality of \eqref{fundamental-identity} opens the path to its application to old and new rigidity results for solutions of the Poisson equation \eqref{torsion}.
In this paper, we shall test the effectiveness of \eqref{fundamental-identity} by analysing the stability of what we shall call the \textit{reverse Serrin problem}, as opposed to the \textit{classical} Serrin problem.  The former consists in considering a solution of \eqref{torsion} subject to the Neumann condition
\begin{equation}
\label{neumann}
u_\nu=R \ \mbox{ on } \ \Ga,
\end{equation}
and then  imposing that $u$ also satisfies \eqref{dirichlet}. 
\par
Of course, the rigidity part of this problem gives the same (spherical) solution of the classical Serrin problem.  The stability part is however not explored.
\par
We shall see in Section \ref{sec:discussion reverse Serrin via Serrin} that one can derive stability estimates for the reverse Serrin problem by exploiting those already obtained for the classical Serrin problem. However, this procedure leads to estimates which require a strong measure of the deviation of $u$ from being constant on $\Ga$. In fact, we need to use the norm $\nr \ol{u} - u \nr_{C^{1,\al}(\Ga)}$, as shown in Theorem \ref{thm:Alternative-stability-reverse-serrin}. 
\par
Instead, the full power of \eqref{fundamental-identity} (in place of \eqref{identity-for-serrin}) allows to weaken such a stringent deviation.
In fact, in the two stability results we shall describe next, the deviation of the function $u$ from being a constant on $\Ga$ will be measured in terms of its \textit{oscillation},
$$
\osc_\Ga u = \max_\Ga u - \min_\Ga u,
$$
and some suitable norm of its \textit{tangential gradient} $\na_\Ga u$ (see Section \ref{sec:preliminary} for its definition).
Our first stability result reads as follows.

\begin{thm}[The reverse Serrin problem with uniform deviation]
	\label{thm:prova}
	Let $\Om \subset \RR^N$, $N \geq 2$, be a bounded domain 
with boundary $\Ga$ of class $C^2$. 
\par
Suppose $u\in C^2(\ol{\Om})$ is a solution of \eqref{torsion}, \eqref{neumann} and let $z$ be any minimum point of $u$ on $\ol{\Om}$.
Then it holds that 
	\begin{equation*}
		\rho_e-\rho_i \le c\,\psi\left( \osc_\Ga u \, + \|\na_\Ga u\|_{\infty,\Ga}\right),
	\end{equation*}
	where $\psi$ is given in \eqref{eq:old-profile-psi}.  
	The constant $c>0$ depends on $N, d_\Om$, $r_i$, and $r_e$. 
\end{thm}

We singled out this theorem since it is relatively easy to obtain from our new identity. 
Moreover, 
the dependence of the constant $c$ on the mentioned parameters
can be written explicitly.
\par
In the next result, we show how the deviation of $u$ from being a constant can be further weakened.

\begin{thm}[The reverse Serrin problem with weak deviation]
\label{thm:stability-reverse-serrin}
Let $\Om \subset \RR^N$, $N \geq 2$, be a bounded domain with boundary $\Ga$ of class $C^{2,\al}$, with $0<\al<1$.
\par
Suppose $u\in C^2(\ol{\Om})$ is a solution of \eqref{torsion}, \eqref{neumann} and let $z$ be any minimum point of $u$ on $\ol{\Om}$.
Then it holds that 
    \begin{equation*} 
        \rho_e-\rho_i \le c\,\psi\left( \osc_\Ga u \, + \|\na_\Ga u\|_{2,\Ga}\right),
    \end{equation*}
where $\psi$ is given in \eqref{eq:old-profile-psi}.  
The constant $c>0$ depends on $N, d_\Om$, $r_i$, and $r_e$. 
\par
Moreover, for $N \ge 4$ the profile $\psi$ can be improved to \eqref{eq:improvementMinE}. In this case, the dependence of $c$ on $r_i$ and $r_e$ 
must be replaced with that on the $C^{2,\al}$-regularity of $\Ga$.
\end{thm}

We shall now describe the key points of the proofs of Theorems \ref{thm:prova} and \ref{thm:stability-reverse-serrin} and highlight 
with the corresponding  results in classical Serrin problem.
\par
The first  step is to rewrite \eqref{fundamental-identity} conveniently, when $u$ also satisfies \eqref{neumann}.
In fact, in this case we obtain the following identity:
\begin{multline}
\label{neumann-identity-intro}
\int_\Om (\ol{u}-u)\,|\na^2 h|^2\,dx=\frac12 \int_\Ga |\na_\Ga u|^2 h_\nu\, dS_x +\\
\int_{\Ga} (\ol{u}-u)\bigl[(N-1)\,R H-N\bigr]h_\nu\,dS_x-\int_\Ga  (\ol{u}-u)\,\lan(\na\nu) \na_\Ga u,\na_\Ga u\ran\,dS_x.
\end{multline}

Here, $h=q-u$, where $q$ is given in \eqref{quadratic}. Also, $H$ is the mean curvature of $\Ga$ and $\na\nu$ is the Jacobian of a suitable extension to a neighborhood of $\Ga$ of the unit normal vector field $\nu$ on $\Ga$. Since up to suitable rotations $\na\nu$ depends on the principal curvatures of $\Ga$, we realize that each surface integral at the right-hand side of \eqref{neumann-identity-intro} contains a quantity which is locally defined on $\Ga$ and can be made small in terms of the relevant deviations used in  
Theorems \ref{thm:prova} and \ref{thm:stability-reverse-serrin}.
\par
In order to obtain the new version \eqref{neumann-identity-intro} from \eqref{fundamental-identity}, terms like $u_\nu$ or $u_{\nu\nu}$ should be replaced in \eqref{fundamental-identity}. 
The latter term pops up when we consider the first surface integral in \eqref{fundamental-identity}.
In order to see that, we extend the vector field $\nu$ to a tubular neighborhood of $\Ga$  and observe that $\na (u_\nu)=\lan\na (u_\nu),\nu\ran\,\nu$, since $\Ga$ is a level surface of $u_\nu$. This procedure leads to discover that 
$$
\bigl\lan (\na^2 u)\, \na u,\nu\bigr\ran=-\bigl\lan (\na \nu)\, \na_\Ga u, \na_\Ga u\bigr\ran+
u_\nu\,u_{\nu \nu} \ \mbox{ on } \ \Ga,
$$
since we also know that $(\na\nu)\, \nu=0$ on $\Ga$ (see Section \ref{sec:preliminary} for details).
\par
Furthermore, the term $u_{\nu\nu}$ is treated by using the well-known formula for the Laplace operator of a smooth function $v$ on a closed surface $\Si$: 
$$
\De v=v_{\nu\nu}+\De_\Si v+(N-1)\,H\,v_\nu.
$$
Here, $\De_\Si$ is the Laplace-Beltrami operator on $\Si$
and $H$ is the mean curvature of $\Si$.
(All these notations and computations are collected in Section \ref{sec:preliminary}.) Thus, since $u$ satisfies \eqref{torsion}  and \eqref{neumann}, we obtain that
$$
u_\nu\,u_{\nu \nu}=N\,R-(N-1)\,R^2 H-R\,\De_\Ga u \ \mbox{ on } \ \Ga.
$$
\par
Finally, if we want obtain the desired dependence on $\na_\Ga u$, we must use integration by parts on $\Ga$, i.e. apply the formula:
$$
\int_\Ga (\ol{u}-u)\,\De_\Ga u\,dS_x=\int_\Ga |\na_\Ga u|^2 dS_x.
$$
This last step and the pointwise analysis on $\Ga$ are peculiar of the reverse Serrin problem and were not needed in the classical counterpart.
\par
Since we assume that $u_\nu=R$ on $\Ga$ the remaining integrals in \eqref{fundamental-identity} are then easily treatable.
\par
Now, we focus our attention on the left-hand side of \eqref{fundamental-identity}. Notice that, if we choose $v=q$,  with $q$ as in \eqref{quadratic}, the second volume integral in \eqref{fundamental-identity} disappears (see also Theorem \ref{th:mother-identity}).
Hence, the left-hand side of \eqref{neumann-identity-intro} is easily obtained by observing that $|\na^2 h|^2=\De P$. 
The harmonic function $h$ was also used in the treatment of \eqref{stability-classical-serrin} and can be viewed as a deviation of $u$ from $q$, thus measuring how far the solution (and the domain) is from a spherical configuration. 
\par
Therefore, in order to obtain the estimate in Theorem \ref{thm:stability-reverse-serrin},
the plan is to show that the volume integral in \eqref{neumann-identity-intro}
is small provided the right-hand side of \eqref{neumann-identity-intro} becomes small in terms of some deviation of $u$ on $\Ga$ from a constant. 
\par
Notice that if that integral is zero, then $|\na^2 h|=0$, being as $\ol{u}-u>0$ by the strong maximum principle. 
Thus, $h$ is an affine function. As a consequence, we obtain in passing a fairly general rigidity result, which holds in the case of fairly general domains.


\begin{thm}
\label{th:quadratic-test}
Let $\Om \subset \RR^N$, $N \geq 2$, be a bounded domain with boundary $\Ga$ of class $C^2$.
Let $u\in C^2(\ol{\Om})$ be a solution of  \eqref{torsion}, \eqref{neumann}. 
\par
If the right-hand side of \eqref{neumann-identity-intro} is non-positive, then $h$ is affine, and hence $u$ is a quadratic polynomial. As a result, $\Om$ is a ball.
\end{thm}
A version of this theorem for the solution of \eqref{torsion}, \eqref{dirichlet-homogeneous} was proved in \cite{MP2}. 
Of course, the condition on the sign of the right-hand side of \eqref{neumann-identity-intro} in Theorem \ref{th:quadratic-test} is satisfied if $u$ is also constant on $\Ga$.
\par
Finally, in order to treat the stability issue, one observes that the affine function $h$ in Theorem \ref{th:quadratic-test} can be further normalized to a constant by an appropriate choice of the center $z$ of the paraboloid $q$. With such a choice, if the volume integral \eqref{neumann-identity-intro} is small, then $h$ is close to a constant. Heuristically, this means that the oscillation of $h$ on $\Ga$ can be controlled by a weighted norm of $|\na^ 2 h|$. Now, since we have that
$$
\frac12\,(\rho^2_e-\rho^2_i)=\osc_\Ga q\le \osc_\Ga h+\osc_\Ga u,
$$
the desired estimate for $\rho_e-\rho_i$ can then be derived from one of the oscillation of $h$ on $\Ga$ in terms of the aforementioned weighted integral, which can be derived from some inequalities proved in \cite{MP5}. 
\par
The structure of this paper is as follows. In Section \ref{sec:preliminary}, we collect some preliminary formulas and results for functions defined on surfaces. Section \ref{sec:identities-rigidity} contains our general identity and its corollaries, together with rigidity results related to the reverse Serrin problem, including the proof of Theorem \ref{th:quadratic-test}. In Section \ref{sec:neumann-estimates}, we gather some estimates for the Neumann problem, which are instrumental to trace the dependence of the constant $c$ of Theorem \ref{thm:stability-reverse-serrin} on the relevant parameters. The proofs of Theorems \ref{thm:prova} and \ref{thm:stability-reverse-serrin} are contained in Section \ref{sec:stability}. Section \ref{sec:discussion reverse Serrin via Serrin} contains the alternative stability result for the Serrin reverse problem (Theorem~\ref{thm:Alternative-stability-reverse-serrin}) obtained via existing results for the classical Serrin problem.

\section{Notations and preliminary formulas}
\label{sec:preliminary}
Let $\Om\subset \RR^N$ be a bounded domain with boundary $\Ga$ of class $C^2$ and consider a function $v\in C^2(\ol{\Om})$.

We use 
the following decomposition for the gradient $\na v$ of $v$ on $\Ga$:
\begin{equation}
\label{gradient-decomposition}
\na v=v_\nu\,\nu+\na_{\Ga} v.
\end{equation}
Here, $\nu$ denotes the exterior unit normal vector field to $\Ga$. This decomposition defines what we shall call the \textit{tangential gradient} $\na_{\Ga} v$ of $v$ on $\Ga$. It is clear that, if $w$ is another function in $C^2(\ol{\Om})$, then
$$
\lan\na v, \na w\ran=v_\nu\,w_\nu+\lan\na_{\Ga} v, \na_{\Ga} w \ran.
$$
\par
We recall the well-known formula (see \cite{HP}):
\begin{equation}
	\label{reilly}
	\De v=v_{\nu\nu}+\De_{\Ga} v+(N-1)\,H\,v_\nu \ \mbox{ on } \ \Ga.
\end{equation}
Here, $\De_{\Ga}$ is the Laplace-Beltrami operator on $\Ga$ and $H$ denotes the \textit{mean curvature} of $\Ga$ defined by
$$
H=\frac1{N-1}\,\sum_{j=1}^{N-1} \ka_j,
$$ 
where $\ka_1, \dots, \ka_{N-1}$ are the principal curvatures of $\Ga$ with respect to the interior normal.
\par
We also recall from \cite{HP} the following formula of integration by parts on $\Ga$, which holds for any $v, w\in C^2(\ol{\Om})$:
\begin{equation}
\label{parts}
\int_{\Ga} \lan\na_{\Ga} v, \na_{\Ga} w \ran\,dS_x=-\int_{\Ga} v\,\De_{\Ga} w\,dS_x.
\end{equation}
In particular, we have that $\int_{\Ga} \De_{\Ga} w\,dS_x=0$.
\par
We can always extend the vector field $\nu$ smoothly to a tubular neighborhood of $\Ga$ in $\ol{\Om}$. For instance, we can set $\nu=-\eta\,\na\de_\Ga$, where $\de_{\Ga}$ is the \textit{distance function}, defined by
$$
\de_{\Ga}(x)=\dist(x,{\Ga}) \ \mbox{ for } \ x\in\ol{\Om},
$$
and $\eta$ is a cut-off function which equals $1$ in a neighborhood of $\Ga$.
In fact, we know from \cite{KP} that there is neighborhood $\cU$ of $\Ga$ in $\ol{\Om}$ such that $\de_\Ga\in C^{2}(\cU)$ and we have that 
$$
|\na\de_{\Ga}|=1 \ \mbox{ in } \ \cU.
$$
Hence, $\eta$ is chosen as a smooth function with support contained in the set in $\cU$ and such that $\eta\equiv 1$ near $\Ga$.
\par
At any point in $\Ga$, we can choose coordinates (see, e.g., \cite[Appendix 14.6]{GT}) so that
$$
\na\nu=-\na^2\de_{\Ga} =
\left[\begin{matrix}
\ka_1  &\cdots &0 &0 \\
\vdots &\ddots &\vdots  &\vdots \\
0  &\cdots &\ka_{N-1} &0 \\
0  &\cdots &0 &0
\end{matrix}
\right].
$$

\medskip

Throughout the rest of the paper,
unless explicitly indicated, $\Om$ is a bounded domain in $\RR^N$, $N\ge 2$, with boundary $\Ga$ of class $C^2$. Under this assumption, $\Om$ surely satisfies the \textit{uniform interior and exterior sphere conditions}, whose respective radii will be denoted by $r_i$ and $r_e$.
Namely, there exists $r_i>0$ (resp. $r_e>0$) such that for each $p \in \Ga$ it holds that $\ol{B} \cap \Ga = \{p\}$ for some ball $B\subset \Om$ (resp. $B \subset \RR^N \setminus\ol{\Om})$ of radius $r_i>0$ (resp. $r_e>0$).
As already mentioned, the two properties are equivalent to the $C^{1,1}$-regularity of $\Ga$, as shown in \cite{ABMMZ}. We denote by $d_\Om$ the diameter of $\Om$.
\par
The following lemma will be useful.

\begin{lem}
\label{lem:normal-modulus-gradient}
Let $u\in C^2(\ol{\Om})$ be a solution of \eqref{torsion}, and $c\in\RR$. It holds that
\begin{enumerate}[(i)]
\item
if $u=c$ on $\Ga$, then 
$$
\bigl\lan \na^2 u\, \na u,\nu\bigr\ran=u_\nu\,[N-(N-1)\,H\,u_\nu] \ \mbox{ on } \ \Ga;
$$
\item
if $u_\nu=c$ on $\Ga$, then
\begin{multline*}
\bigl\lan \na^2 u\, \na u,\nu\bigr\ran=-\lan \na\nu\, \na_\Ga u,\na_\Ga u\ran + u_\nu\,[N-\De_\Ga u-(N-1)\,H\,u_\nu] \ \mbox{ on } \ \Ga.
\end{multline*}
\end{enumerate}
\end{lem}

\begin{proof}
(i) Since $\Ga$ is a level surface for $u$, by \eqref{gradient-decomposition} we infer that
$\na u=u_\nu\,\nu$, and hence
$$
\bigl\lan \na^2 u\, \na u,\nu\bigr\ran=u_\nu\,u_{\nu \nu}.
$$
Moreover, \eqref{reilly} with $v=u$ gives that
\begin{equation}
\label{u-nu-nu}
u_{\nu\nu}=
N-\De_{\Ga}u-(N-1)\,H\,u_\nu \ \mbox{ on } \ \Ga.
\end{equation}
Since $\De_{\Ga}u=0$ on $\Ga$, the desired identity ensues at once.
\par
(ii) We proceed as already described by extending $\nu$ to a neighborhood $\cU$.
The ensuing computations are generally performed in $\cU$, and hence evaluated on $\Ga$.
\par
Since $\Ga$ is a level surface of $\lan\na u,\nu\ran$, we have that
$$
\na \lan\na u,\nu\ran=\lan \na \lan\na u,\nu\ran, \nu\ran\,\nu
$$
on $\Ga$. Next, we compute that
\begin{multline*}
\na \lan\na u,\nu\ran=(\na^2 u)\, \nu+\na \nu\, \na u= \\
(\na^2 u)\, \nu+u_\nu (\na \nu)\,\nu+\na \nu\,\na_\Ga u= 
(\na^2 u)\, \nu+\na \nu\,\na_\Ga u.
\end{multline*}
Here, we have used that $(\na \nu)\,\nu=0$ in $\cU$, since $\nu$ is unitary. 
\par
Thus, we infer that
\begin{multline*}
(\na^2 u)\, \nu+(\na \nu)\,\na_\Ga u= \\
\bigl[u_{\nu \nu}+\lan(\na \nu)\,\na_\Ga u,\nu\ran\bigr] \nu= 
\bigl[u_{\nu \nu}+\lan(\na \nu)\,\nu,\na_\Ga u\ran\bigr] \nu=
u_{\nu \nu}\, \nu.
\end{multline*}
Hence, multiplying by $\na u$ gives that
\begin{multline*}
\bigl\lan (\na^2 u)\, \na u,\nu\bigr\ran= \\
-\bigl\lan (\na \nu)\, \na_\Ga u, \na u\bigr\ran+
u_\nu\,u_{\nu \nu}=-\bigl\lan (\na \nu)\, \na_\Ga u, \na_\Ga u\bigr\ran+
u_\nu\,u_{\nu \nu}.
\end{multline*}
Here, we have again used \eqref{gradient-decomposition} and $(\na\nu)\,\nu=0$. Therefore, the desired identity follows by \eqref{u-nu-nu}.
\end{proof}

\begin{rem}\label{rem: kappa min max}
{\rm
From the aforementioned properties of $\na\nu$, we easily infer that
$$
\ul{\ka}\,|\na_\Ga u|^2\le\bigl\lan (\na \nu)\, \na_\Ga u, \na_\Ga u\bigr\ran\le\ol{\ka}\,|\na_\Ga u|^2 \ \mbox{ on } \Ga, 
$$
where
$$
\ul{\ka}=\min\bigl[\ka_1, \dots, \ka_{N-1}\bigr] \mbox{ and } \ \ol{\ka}=\max\bigl[\ka_1, \dots, \ka_{N-1}\bigr]
$$
at each point of $\Ga$.
We also recall that
\begin{equation*}
-\frac1{r_e}\le\ul{\ka}\le\ol{\ka}\le\frac1{r_i} \ \mbox{ on } \ \Ga,
\end{equation*}
where $r_i$ and $r_e$ are the radii of the uniform interior and exterior sphere conditions.
}
\end{rem}

\section{Integral identities and  rigidity results}
\label{sec:identities-rigidity}
In this section, we shall derive the identities \eqref{fundamental-identity} and \eqref{neumann-identity-intro}. We will also prove Theorem \ref{th:quadratic-test}.

\begin{thm}[General identity]
\label{th:fundamental-identity}
Suppose $\Om\subset\RR^N$, $N\ge 2$, is a bounded domain with boundary $\Ga$ of class $C^2$.
Let $u, v\in C^2(\ol{\Om})$ be solutions of \eqref{torsion}.
\par
Set
$$
\ol{u}=\max_{\ol{\Om}} u=\max_{\Ga}u,
$$
and
\begin{equation*}
P=\frac12\,|\na u|^2+(\ol{u}-u) \ \mbox{ on }  \ \ol{\Om}.
\end{equation*}
Then the identity \eqref{fundamental-identity} holds, that is:
\begin{multline*}
\int_\Om (\ol{u}-u)\,\De P\,dx +  \int_\Om \lan (I - \na^2 v) \na u, \na u\ran \,dx =\\
\int_\Ga (\ol{u}-u)\,\lan\na^2 u\, \na u,\nu\ran\,dS_x +
\frac12\,\int_\Ga |\na u|^2 (u_\nu+ v_\nu)\,dS_x+ \\
-\int_\Ga \lan \na v,\na u \ran \,u_\nu\,dS_x 
-\int_\Ga (\ol{u}-u)\,u_\nu\,dS_x
+
N \int_\Ga (\ol{u}-u) (u_\nu - v_\nu) \,dS_x.    
\end{multline*}
\end{thm}

\begin{proof}
The proof proceeds by direct computations and an application of the divergence theorem. By simply taking derivatives and by straightforward algebraic manipulations, we first infer that the following key differential identity holds:
\begin{multline*}
(\ol{u}-u)\,\De P + \lan (I - \na^2 v) \na u, \na u\ran = \\
\dv\left\{P\, \na u+(\ol{u}-u)\,\na P+\frac12\,|\na u|^2\, \na v - \lan \na v,\na u \ran \,\na u\right\} 
\\
+\dv\left\{(N-1)\,(\ol{u}-u)\,\na u-N\,(\ol{u}-u)\,\na v\right\}.
\end{multline*} 
Thus, an application of the divergence theorem yields:
\begin{multline*}
    \int_\Om (\ol{u}-u)\,\De P\,dx +  \int_\Om \lan (I - \na^2 v) \na u, \na u\ran \,dx= \\
\int_{\Ga}\Bigl\{ P u_{\nu} +(\ol{u}-u)\, P_{\nu} + \frac12\,|\na u|^2\, v_{\nu} - \lan \na v, \na u \ran u_{\nu}\Bigr\}dS_x+ \\
\int_{\Ga}\bigl\{ (N-1)\,(\ol{u}-u) u_{\nu} - N\,(\ol{u}-u)\, v_{\nu}\bigr\}dS_x.
\end{multline*}
Hence, we infer that
\begin{multline*}
 \int_\Om (\ol{u}-u)\,\De P\,dx +  \int_\Om \lan (I - \na^2 v) \na u, \na u\ran \,dx= 
\int_{\Ga} \Bigl\{\frac12\,|\na u|^2+ (\ol{u}-u)\Bigr\} u_{\nu}\,dS_x +\\ 
\int_\Ga\Bigl\{ \lan \na^2 u\,\na u, \nu \ran- u_{\nu} \Bigr\}(\ol{u}-u)\,  dS_x+
 \int_{\Ga} \Bigl\{\frac12\,|\na u|^2\, v_\nu - \lan \na v, \na u\ran \,u_\nu\Bigr\} dS_x+\\
    \int_{\Ga} \bigl\{(N-1)\,(\ol{u}-u) u_{\nu} - N\,(\ol{u}-u)\,v_\nu\bigr\} dS_x.
\end{multline*}
\par
Identity  \eqref{fundamental-identity} then ensues by rearranging the terms.
\end{proof}

We now present a corollary of the previous result, when $v$ is a  quadratic solution $q$ of \eqref{torsion} as defined in \eqref{quadratic}.

\begin{cor}[Mother identity for the quadratic case]
\label{th:mother-identity}
Suppose $\Om\subset\RR^N$, $N\ge 2$, is a bounded domain with boundary $\Ga$ of class $C^2$.
\par
Let $u$, $\ol{u}$, and $P$ be as in Theorem \ref{th:fundamental-identity} and let $q$ be a quadratic polynomial of the form \eqref{quadratic}.
Then, the following identity holds:
\begin{multline}
\label{mother-identity}
\int_\Om (\ol{u}-u)\,\De P\,dx = \\
\frac12\,\int_\Ga |\na u|^2 \bigl(u_\nu+ q_\nu\bigr) dS_x
-\int_\Ga \lan \na q,\na u \ran \,u_\nu\,dS_x + \\
\int_\Ga (\ol{u}-u)\,\lan\na^2 u\, \na u,\nu\ran\,dS_x +\int_\Ga (\ol{u}-u)\bigl[(N-1)\,u_\nu-N\,q_\nu\bigr] dS_x.
\end{multline}
\par
An alternative formula is:
\begin{multline}
\label{mother-identity-2}
\int_\Om (\ol{u}-u)\,\De P\,dx = 
\frac12\,\int_\Ga u_\nu^2 \bigl(u_\nu- q_\nu\bigr) dS_x+ \\
\frac12\,\int_\Ga |\na_\Ga u|^2 \bigl(u_\nu+ q_\nu\bigr) dS_x
-\int_\Ga \lan \na_\Ga q,\na_\Ga u \ran \,u_\nu\,dS_x + \\
\int_\Ga (\ol{u}-u)\,\lan\na^2 u\, \na u,\nu\ran\,dS_x +\int_\Ga (\ol{u}-u)\bigl[(N-1)\,u_\nu-N\,q_\nu\bigr] dS_x.
\end{multline}
\end{cor}

\begin{proof}
We apply \eqref{fundamental-identity} with $v=q$, since $q$ is clearly a solution of \eqref{torsion} on the whole $\RR^N$. We have that $\na^2 q=I$. 
As a result, the second volume integral at the left-hand side of \eqref{fundamental-identity} vanishes,
and hence \eqref{mother-identity} ensues. Formula \eqref{mother-identity-2} then simply follows by applying \eqref{gradient-decomposition} and rearranging some terms.
\end{proof}

\par
Next, we consider the case in which $u_\nu$ is constant on $\Ga$ with no constraint on the values of $u$ on $\Ga$. 

\begin{cor}[Constant Neumann condition]
\label{cor:locally-neumann}
Suppose $\Om\subset\RR^N$, $N\ge 2$, is a bounded domain with boundary $\Ga$ of class $C^2$.
\par
Let $q$ be as in \eqref{quadratic} and set $h=q-u$, where $u\in C^2(\ol{\Om})$ is a solution of \eqref{torsion}. 
Further,
suppose that $u$ satisfies \eqref{neumann}.
Then the identity \eqref{neumann-identity-intro} holds, that is:
\begin{multline*}
\int_\Om (\ol{u}-u)\,|\na^2 h|^2\,dx=\frac12 \int_\Ga |\na_\Ga u|^2 h_\nu\, dS_x +\\
\int_{\Ga} (\ol{u}-u)\bigl[(N-1)\,R H-N\bigr]h_\nu\,dS_x-\int_\Ga  (\ol{u}-u)\,\lan(\na\nu) \na_\Ga u,\na_\Ga u\ran\,dS_x.
\end{multline*}

\end{cor}

\begin{proof}
We work on \eqref{mother-identity-2}. As already observed, we compute that $\De P=|\na^2 h|^2$.
Next, we compute the surface integrals at the right-hand side.
We begin with
$$
\int_\Ga u_\nu^2 \bigl(u_\nu- q_\nu\bigr) dS_x=R^2\int_\Ga \bigl(u_\nu- q_\nu\bigr) dS_x=0,
$$
thanks to the divergence theorem. Next, we compute that
\begin{multline*}
\frac12\,\int_\Ga |\na_\Ga u|^2 \bigl(u_\nu+ q_\nu\bigr) dS_x
-\int_\Ga \lan \na_\Ga q,\na_\Ga u \ran \,u_\nu\,dS_x= \\
\frac12\,\int_\Ga |\na_\Ga u|^2 \bigl(R+ q_\nu\bigr) dS_x-
R \int_{\Ga} (\ol{u}-u)\,\De_\Ga q\,dS_x.
\end{multline*}
Here, we used \eqref{neumann} and \eqref{parts}, in the second integral.
\par
Moreover, we infer that
\begin{multline*}
\int_{\Ga} (\ol{u}-u)\,\De_\Ga q\,dS_x= 
-(N-1)\int_{\Ga} (\ol{u}-u)\,H\,q_\nu\,dS_x +
(N-1)\,\int_\Ga (\ol{u}-u)\,dS_x.
\end{multline*}
This follows from the fact that
taking $v=q$ in \eqref{reilly} gives:
\begin{equation*}
\De_{\Ga} q=
(N-1)\,(1-H\,q_\nu\bigr) \ \mbox{ on } \ \Ga.
\end{equation*}
\par
Now, we apply item (ii) of Lemma \ref{lem:normal-modulus-gradient} and get:
\begin{multline*}
\int_\Ga (\ol{u}-u)\,\lan\na^2 u\, \na u,\nu\ran\,dS_x= 
-\int_\Ga  (\ol{u}-u)\,\lan(\na\nu) \na_\Ga u,\na_\Ga u\ran\,dS_x+ \\
R\int_{\Ga} (\ol{u}-u)\,
[N-\De_\Ga u-(N-1)\,R H]\,dS_x.
\end{multline*}
Moreover, thanks to \eqref{parts} we compute that
$$
\int_{\Ga} (\ol{u}-u)\,
\De_\Ga u\,dS_x=\int_{\Ga} |\na_\Ga u|^2 dS_x.
$$
Thus, we can write that
\begin{multline*}
\int_\Ga (\ol{u}-u)\,\lan\na^2 u\, \na u,\nu\ran\,dS_x= -\int_\Ga  (\ol{u}-u)\,\lan(\na\nu) \na_\Ga u,\na_\Ga u\ran\,dS_x+ \\
-(N-1)\,R^2\int_{\Ga} (\ol{u}-u)\,H\,dS_x
-R \int_{\Ga} |\na_\Ga u|^2 dS_x+
N\,R\int_\Ga (\ol{u}-u)\,dS_x.
\end{multline*}
\par
Summing up all these formulas then gives \eqref{neumann-identity-intro}, after straightforward algebraic manipulations.
\end{proof}

\par
We now present a quadratic-radial symmetry test. Its analogous for the solution of \eqref{torsion}, \eqref{dirichlet-homogeneous} essentially provides the spherical detector used in \cite{We, PS, BNST, MP1, MP2} (see also \cite[Lemma 1.9]{PogTesi} for additional details). Here, we rather focus on solutions of \eqref{torsion}, \eqref{neumann}. For comparison, we also provide (in item $(i)$) a version of its analogous for the solution of \eqref{torsion}, \eqref{dirichlet}.

\begin{lem}[Radial symmetry test]
\label{lem:quadratic-test}
Let $\Om$ be a bounded domain in $\RR^N$ with boundary $\Ga$ of class $C^2$.
\par
Let $u\in C^2(\ol{\Om})$ be a solution of \eqref{torsion} satisfying either \eqref{dirichlet} or \eqref{neumann}.
If $\De P\equiv 0$ in $\Om$, then $u$ is a quadratic polynomial and $\Ga$ is a sphere of radius $R$ given by \eqref{def-R}.
\end{lem}

\begin{proof}
We know that
$$
\De P=|\na^2 u|^2 - \frac{(\De u)^2}{N} = |\na^2 u|^2  - \frac{\lan \na^2 u, I\ran^2}{N}\ge 0.
$$
Here, $\lan \cdot,\cdot \ran$ denotes the inner product in $\RR^{N^2}$ and hence the last inequality follows by the Cauchy-Schwarz inequality. Thus, the fact that $\De P\equiv 0$ gives that the $N^2$-vectors $\na^2 u$ and $I$ coincide in $\Om$, being as \eqref{torsion} in force. 
Therefore, $u$ is a quadratic polynomial of the form \eqref{quadratic}
for some $z\in\RR^N$ and $a\in\RR$ (see also \cite{MP1, MP2} and \cite[Lemma~1.9]{PogTesi}). 
\par
(i) If \eqref{dirichlet} holds then we have that $|x-z|^2=2\,(u_0-a)$ for every $x\in\Ga$. Hence, $\Ga$ must be a sphere of radius $\sqrt{2\,(u_0-a)}$ and we have that $x-z=\sqrt{2\,(u_0-a)}\,\nu(x)$. Thus, we obtain that
\begin{multline*}
2\,(u_0-a)\,|\Ga|=\int_\Ga |x-z|^2 dS_x= \\
\sqrt{2\,|u_0-a|} \int_\Ga \lan x-z,\nu(x) \ran \,dS_x=
N\, |\Om|\,\sqrt{2\,|u_0-a|}.
\end{multline*}
Consequently, \eqref{def-R} gives that $2\,(u_0-a)=R^2$, 
$$
u(x)=\frac12\,(|x-z|^2-R^2)+u_0,
$$
and $\Ga$ is a sphere centered at $z$ of radius $R$.
\par
(ii) If \eqref{neumann} holds, since
$u$ is a quadratic polynomial as in \eqref{quadratic},
we infer that
\begin{equation}
\label{Gamma-j-constraint}
\lan x-z, \nu(x) \ran=u_\nu(x) = R \ \ \mbox{ for } \  x \in \Ga.
\end{equation}
\par
Next, let $x^m$  and $x^M$ be points in $\Ga$ such that
$$
|x^m-z|=\min_{x\in\Ga}|x-z| \ \mbox{ and } \ |x^M-z|=\max_{x\in\Ga}|x-z|.
$$
It is clear that $x^m-z$ and $x^M-z$ are parallel to $\nu(x^m)$ and $\nu(x^M)$.  Hence, \eqref{Gamma-j-constraint} gives that  $|x^m-z|=|x^M-z|=R$, i.e. $|x-z|=R$ for every $x\in\Ga$.
Thus, again, $\Ga$ must be the sphere centered at $z$ of radius $R$.
\end{proof}

A consequence of item (ii) of this lemma and Corollary \ref{cor:locally-neumann} is Theorem \ref{th:quadratic-test}, the general rigidity result stated in the Introduction.

\begin{proof}[Proof of Theorem \ref{th:quadratic-test}]
The assumption on the right-hand side of \eqref{neumann-identity-intro} forces the function $|\na^2 h|^2=\De P$ to be identically zero, being as $\ol{u}-u>0$ by the strong maximum principle. Our claim then follows from Lemma \ref{lem:quadratic-test}.
\end{proof}

\section{Estimates for solutions of the Neumann problem}

\label{sec:neumann-estimates}
In this section, we derive some geometric and spectral estimates for solutions of \eqref{torsion}, \eqref{neumann}, which will be useful to trace the dependence of the constant $c$ in Theorems \ref{thm:prova} and \ref{thm:stability-reverse-serrin}.

We begin with a geometric inequality that adapts \cite[Lemma 3.1]{MP2}, obtained for the solution of \eqref{torsion}, \eqref{dirichlet}.

\begin{lem}
\label{lemma geo bound -u > de^2}
Let $\Om$ be a bounded domain in $\RR^N$.  Let $u\in C^2(\ol{\Om})$ be a solution of \eqref{torsion}.
Then, it holds that
\begin{equation*}
    \ol{u} - u(x) \ge \frac12\, \de_\Ga(x)^2 \ \mbox{ for every } \ x \in\ol{\Om}.
\end{equation*}
Moreover, if $\Ga$ is of class $C^1$ and satisfies the uniform interior sphere condition with radius $r_i$, then 
\begin{equation*}
    \ol{u} - u(x) \ge \frac12\,r_i\,\de_\Ga(x) \ \mbox{ for every } \ x \in \ol{\Om}.
\end{equation*}
\end{lem}
\begin{proof}
Let us define by $f$ the unique solution of 
\begin{equation*}
    \De f=N \ \mbox{ in } \ \Om, 
    \qquad
    f = 0 \ \text{ on } \ \Ga.
\end{equation*}
Then, by comparison $\ol{u} - u \ge -f$ in $\ol{\Om}$, and the conclusion follows by \cite[Lemma 3.1]{MP2}.
\end{proof}

\medskip

We shall now derive a priori estimates on the harmonic function $h=q-u$ and its derivatives, 
in terms of geometric, spectral, and regularity parameters of the domain $\Om$. 
\par
In what follows, we shall use the fact that a function $u$ satisfying \eqref{torsion} and \eqref{neumann} minimizes the functional $\cE: W^{1,2}(\Om)\to\RR$ defined by 
$$
\cE(v) = \frac12 \int_\Om |\na v|^2\,dx + N \int_\Om v \,dx - R \int_\Ga v \,dS_x,
$$
for any $v\in W^{1,2}(\Om)$.
In the following result, $\nu_2(\Om)$, $\si_2(\Om)$, 
denote the second (non-trivial) Neumann and Steklov eigenvalues. Respectively, they are defined by
$$
\nu_2(\Om)=\min\left\{\int_\Om |\na v|^2 dx: v\in W^{1,2}(\Om), \int_\Om v^2 dx=1, \int_\Om v\,dx=0\right\}.
$$
and
$$
\si_2(\Om)=\min\left\{\int_\Om |\na v|^2 dx: v\in W^{1,2}(\Om), \int_\Ga v^2 dS_x=1, \int_\Ga v\,dS_x=0\right\}.
$$

\begin{prop}[$L^2$ bound for oscillation]\label{prop:L^2 bound for oscillation}
Let $\Om$ be a bounded domain in $\RR^N$ with boundary $\Ga$ of class $C^2$. 
Let $u\in C^2(\ol{\Om})$ be a solution of \eqref{torsion}, \eqref{neumann} and set $h=q-u$, with $q$ given by \eqref{quadratic}.
Then, it holds that
\begin{equation}
\label{bound-L2}
\nr h-h_\Om\nr_{2,\Om}\le \frac{2\,\nr R-q_\nu\nr_{2,\Ga}}{\sqrt{\nu_2(\Om)\,\si_2(\Om)}}.
\end{equation}
Here, $h_\Om$ denotes the mean value of $h$ on $\Om$.
\end{prop}
\begin{proof}
Since $h$ is harmonic in $\Om$ and $h_\nu=q_\nu-R$ on $\Ga$, it is clear that $h$ minimizes in $W^{1,2}(\Om)$ the functional defined by
$$
W^{1,2}(\Om)\ni v\mapsto\frac12 \int_\Om |\na v|^2\,dx -\int_\Ga v\, (q_\nu-R)\,dS_x.
$$
By taking $v\equiv 1$, the minimality of $h$ then gives that
$$
\int_\Om |\na h|^2\,dx\le 2\,\int_\Ga h\, (q_\nu-R)\,dS_x=2\,\int_\Ga (h-h_\Ga)\, (q_\nu-R)\,dS_x,
$$
where $h_\Ga$ is the mean value of $h$ on $\Ga$. Here, we used that $q_\nu-R$ has a zero mean value on $\Ga$. The Cauchy-Schwarz inequality and the definition of $\si_2(\Om)$  then give that
$$
\int_\Om |\na h|^2\,dx\le 2\,\nr q_\nu-R\nr_{2,\Ga} \nr h-h_\Ga\nr_{2,\Ga} \le
\frac{2\,\nr q_\nu-R\nr_{2,\Ga}}{\sqrt{\si_2(\Om)}}\,\nr \na h\nr_{2,\Om},
$$  
and hence
$$
\nr h-h_\Om\nr_{2,\Om}^2\le \nu_2(\Om)^{-1} \int_\Om |\na h|^2\,dx\le \frac{4\,\nr q_\nu-R\nr_{2,\Ga}^2}{\nu_2(\Om)\,\si_2(\Om)},
$$
by the Poincar\'e inequality. Thus, \eqref{bound-L2} ensues.
\end{proof}

\begin{rem}\label{rem:uniform Poincare inequalities}
{\rm 
Thanks to \cite[Theorems 1 and 2]{BC}, we have that, within the class of uniformly bounded and uniformly Lipschitz domains, $\nu_2(\Om)$ and $\si_2(\Om)$ are bounded from below by a positive constant independent of $\Om$.
} 
\end{rem}

\begin{thm}[Uniform bound for the derivatives of $h$]
\label{th:bound-derivatives}
Set $m=1$ or $m=2$.
Let $\Om$ be a bounded domain in $\RR^N$ of class $C^{m,\al}$, $0 < \al < 1$.
Let $u$ be a solution of \eqref{torsion}, \eqref{neumann} and set $h=q-u$, with $q$ given by \eqref{quadratic}. Then, there exists a constant $C$ that depends on $N$, $d_\Om$, and the $C^{m,\al}$ regularity of $\Ga$, such that
$$
\nr h-h_\Om\nr_{m,\al;\,\Om}\le C ,
$$
were $\nr \cdot \nr_{m,\al;\,\Om}$ denotes the norm in $C^{m,\al}(\Om)$.
\end{thm}
\begin{proof}
For simplicity of notation, in this proof by the same letter $c$ we shall denote a generic constant, whose dependence will be specified if needed.
\par
We have that $h$ is harmonic in $\Om$ and $h_\nu=q_\nu-R$ on $\Ga$. We thus can apply a standard a priori estimate for the Neumann problem  (see e.g., \cite[Chapter 7]{Li}; see also\cite[Theorem 6.30]{GT} for $m=2$) to obtain that
\begin{equation}
\label{eq:standard elliptic regularity}
\|h\|_{m,\al;\,\Om} \le c \left(\max_{\ol{\Om}}|h-h_\Om|+ \|q_\nu-R\|_{m-1,\al;\,\Ga}\right).
\end{equation}
Here, $c$ depends on $N, m, \al$, $d_\Om$, and the $C^{m,\al}$ regularity of $\Ga$. 
By arguing as in \cite[Lemma 4.9]{MP6} (see also \cite{MPchapter}) we can find a  constant $c$ such that
\begin{equation}
\label{eq:stima interpolatoria PogTesi}
\max_{\ol{\Om}} |h-h_\Om| = \max_{\Ga} | h - h_\Om | \le c \left\lbrace \frac{\nr h-h_\Om \nr_{2,\Om}}{\si^{N/2}} + \si \,  \nr \na h \nr_{\infty, \Om }   \right\rbrace ,
\end{equation}
for any $0< \si < \si_0$. Here, $c$ and $\si_0$ only depend on  $N$, $d_\Om$, and the $C^{m,\al}$-regularity of $\Ga$.\footnote{This argument gives a constant depending on $N$ and the parameters related to a uniform interior cone condition. This condition is equivalent to a Lipschitz-type regularity of $\Ga$ (see \cite{MP6} for details). Of course, the relevant parameters can be bounded in terms of the $C^{1,\al}$ regularity of $\Ga$.}
\par
Now, 
if we choose $\si= \min\{1/(2 c^2) ,\si_0\}$ (here, $c$ is the maximum of the two constants appearing in \eqref{eq:standard elliptic regularity} and \eqref{eq:stima interpolatoria PogTesi}) and we plug \eqref{eq:stima interpolatoria PogTesi} into \eqref{eq:standard elliptic regularity}, 
we easily find that
\begin{equation*}
\|h\|_{m,\al;\,\Om} \le c \left( \nr h-h_\Om \nr_{2,\Om} + \|q_\nu-R\|_{m-1,\al;\,\Ga}\right),
\end{equation*}
by using the trivial inequality 
$
\nr \na h \nr_{\infty, \Om }  \le \|h\|_{m,\al;\,\Om}.
$
Here, $c$ can be explicitly derived in terms of $\si_0$ and the constants in \eqref{eq:standard elliptic regularity} and \eqref{eq:stima interpolatoria PogTesi}.

The term $\|q_\nu-R\|_{m-1,\al;\,\Ga}$ can be easily bounded by a constant depending on the same relevant parameters. Finally, 
Proposition \ref{prop:L^2 bound for oscillation} and Remark \ref{rem:uniform Poincare inequalities} ensure that the norm $\nr h-h_\Om \nr_{2,\Om}$ can be bounded by a constant depending on the same relevant parameters. 
The desired conclusion then ensues.
\end{proof}

\section{Quantitative symmetry results}
\label{sec:stability}

In this section, based on the identity \eqref{neumann-identity-intro}, we shall provide our stability estimates for the radial symmetry in the reverse Serrin overdetermined problem (Theorems \ref{thm:prova} and \ref{thm:stability-reverse-serrin}). 
To this aim, 
for a point $z \in \Om$, we recall the definition \eqref{def-rhos}:
\begin{equation*}
    \rho_i = \min_{x \in \Ga} |x-z| \quad \text{and} \quad \rho_e = \max_{x \in \Ga} |x-z|.
\end{equation*}
It is clear that  $\rho_i$ is the radius of the largest ball contained in $\Om$ and $\rho_e$ is the radius of the smallest ball that contains $\Om$, both balls centered at $z$. 
\par
Thus, for a solution $u$ of \eqref{torsion}, \eqref{neumann}, we want to bound the difference $\rho_e-\rho_i$ in terms of the sum of the oscillation of $u$ on $\Ga$ and a suitable Lebegue norm of $|\na_\Ga u|$ on $\Ga$.

\begin{rem}
{\rm
Let $u$ be a solution of \eqref{torsion} , \eqref{neumann}. A convenient choice for the point $z$ in \eqref{def-rhos} is any  minimum point of $u$ in $\ol{\Om}$. In fact, such a point must fall in $\Om$, since otherwise $u_\nu\le 0$ at that point, contrary to \eqref{neumann}.
}
\end{rem}

\par
We shall first adapt to our case some estimates on the oscillation of a harmonic function, which were derived in \cite{MP5}.

\subsection{Interpolating estimates for Sobolev functions}

We now recall three key inequalities for the oscillation of a harmonic function.  These descend from some interpolation inequalities proved in \cite[Theorem 4.1]{MP5} (see also \cite[Appendix]{MP6}). Here, we present these inequalities for $C^2$ domains, as it is enough for the purposes of the present paper. Nevertheless, the original theorem was stated and proved for a bounded domain satisfying the $(\theta,a)$-uniform cone condition (see \cite[Appendix]{MP6} for a definition). 

\begin{lem}
\label{thm osc bound}
    Let $\Om \subset \RR^N$, $N \geq 2$, be a bounded domain with boundary $\Ga$ of class $C^2$, and let $u$ be solution of \eqref{torsion}, \eqref{neumann}.
\par
Let $h = q^z-u$, where $q^z$ is that in \eqref{quadratic} and $z$ is any global minimum point of $u$ in $\ol{\Om}$. Then, we have that
    \begin{equation*}
        \osc_\Ga h \le c 
        \begin{cases}
            \|\sqrt{\de_\Ga}\, \na^2 h \|_{2,\Om} &\ \mbox{ if } \ N=2 ,
            \\
            \|\sqrt{\de_\Ga} \, \na^2 h \|_{2,\Om}  \max \left[\log \left( \frac{e \|\na h \|_{\infty,\Omega}}{\|\sqrt{\de_\Ga}\, \na^2 h \|_{2,\Om}} \right), 1 \right] &\ \mbox{ if } \ N=3 ,
            \\
            \| \na h \|^{\frac{N-3}{N-1}}_{\infty,\Om}\,
            \|\sqrt{\de_\Ga} \, \na^2 h \|_{2,\Om}^{\frac{2}{N-1}}
            & \ \mbox{ if } \ N\ge 4 ,
        \end{cases}
    \end{equation*}
where $c$ is an explicit constant that only depends on $N, d_\Om, r_i$, and a lower bound for $\de_\Ga(z)$.
\end{lem}
\begin{proof}
By definition of $h$ and the choice of $z$, we have that $h$ is harmonic and that $\na h(z)=0$ at the point $z\in\Om$.
	
The desired result is essentially contained in \cite[Theorem 4.1]{MP5}. To be precise, that theorem gives an estimate for the difference $\rho_e-\rho_i$, which is there related to the oscillation of the particular harmonic function $v=q-f$, where $f$ is the solution of the Dirichlet problem \eqref{torsion}, \eqref{dirichlet-homogeneous} and $q$ is given in \eqref{quadratic}. Indeed there holds that
$$
\max_{\Ga} v - \min_{\Ga} v=\frac12\,(\rho_e^2-\rho_i^2)
$$ 
and
\begin{equation}
\label{bound-h-rho}
\rho_e^2-\rho_i^2\ge\left(\frac{|\Om|}{|B|}\right)^\frac1{N}(\rho_e-\rho_i).
\end{equation}
\par
Now, if $u$ is a solution of the Neumann problem \eqref{torsion}, \eqref{neumann}, instead, $u$ is no longer constant on $\Ga$. Nevertheless, a careful inspection reveals that the inequalities contained in \cite[Theorem 4.1]{MP5} concern the oscillation on $\Ga$ of any harmonic function in $\Om$, whose gradient vanishes at a given point $z$ in $\Om$, and that the constant $c$ only depends on $N$, $d_\Om$, $r_i$, and a lower bound of the distance $\de_\Ga (z)$ of $z$ to $\Ga$. 
Thus, the desired result follows.
\end{proof}

\begin{rem}\label{rem:possibilestima de(z)}
{\rm
By adapting the arguments of \cite[Remark 2.9]{MP3} (see also \cite{MP4} for generalizations) to the present setting, we can choose $x\in \Om$ such that $\de_{\Ga}(x)=r_i$, $y\in \Ga$ such that $|y-z|=\de_\Ga(z)$ and, recalling Lemma \ref{lemma geo bound -u > de^2}, compute that 
\begin{multline*}
	\frac{r_i^2}{2} 
	\le \ol{u} - u (x) 
	\le \ol{u} + \max_{\ol{\Om}}(-u) 
	 = \ol{u} - u(z) =
	\\
	 \ol{u} - u(y) + u(y) - u(z) 
	 \le \osc_{\Ga} u + \nr \na u \nr_{\infty , \Om} \,  \de_\Ga (z) ,
\end{multline*}
from which we get that
$$
\de_\Ga (z) \ge \frac{r_i^2 - 2 \osc\limits_\Ga u }{2 \nr \na u \nr_{\infty , \Om} }.
$$
}
\end{rem}

In the following lemma, we use the tangential norm defined by 
\begin{equation*}
    \nr \ol{u} - u \nr_{1,2,\tau,\Ga}^2 = \nr \ol{u} - u \nr^2_{2,\Ga} + \nr \na_\Ga u \nr^2_{2,\Ga}.
\end{equation*}

\begin{lem}\label{lem bound in W_1,2,Ga}
    Let $\Om \subset \RR^N$, $N \geq 2$, be a bounded domain with boundary $\Ga$ of class $C^2$, and let $u$ be a solution of \eqref{torsion}, \eqref{neumann}.
\par
Let $h = q-u$, where $q$ is that in \eqref{quadratic} and $z$ is any global minimum point of $u$ in $\ol{\Om}$. 
Then, we have that
\begin{equation*}
    \int_\Om (\ol{u}-u)\,|\na^2 h|^2\,dx \le c\Bigl[1+\osc_{\Ga} u\Bigr] \nr \ol{u} - u \nr_{1,2,\tau,\Ga}^2,
\end{equation*}
where $c$ is an explicit constant that only depends on $N, d_\Om, r_i$, $r_e$, and a lower bound for $\de_\Ga(z)$.
\end{lem}

\begin{proof}
For simplicity of notation, in this proof by the same letter $c$ we shall denote a generic constant, whose dependence will be specified if needed.
\par
We start working on the right-hand side of \eqref{neumann-identity-intro}. 
By Remark \ref{rem: kappa min max} and the Holder inequality, we see that there is a constant $c$, only depending on $N, d_\Om, r_e$ and $r_i$, such that
\begin{multline*}
    \int_\Om (\ol{u}-u)\,|\na^2 h|^2\,dx \le \\
c\biggl[\int_\Ga |\na_\Ga u|^2\,dS_x
    +\|h_\nu\|_{2,\Ga} \nr \ol{u} - u \nr_{2,\Ga}+\nr \ol{u} - u \nr_{\infty,\Ga} 
\int_\Ga |\na_\Ga u|^2\,dS_x \biggr].
\end{multline*}
Here, we have used the fact that $|h_\nu| \le d_\Om + R$ on $\Ga$.
\par
Next, let us define by $f$ the unique solution of 
\begin{equation*}
    \De f=N \ \mbox{ in } \ \Om, 
    \qquad
    f = 0 \ \text{ on } \ \Ga.
\end{equation*}
By \cite[Lemma 2.5]{MP3} applied with $v=h$, we obtain that
\begin{equation*}
\int_\Ga h_\nu^2\, dS_x \le  \int_\Ga |\na h|^2 dS_x \le c \int_\Om (-f)\,|\na^2 h|^2 dx
\end{equation*}
where $c$ is an explicit constant that only depends on $N, d_\Om, r_i,$ and a lower bound for $\de_\Ga(z)$.
Since, by comparison, $\ol{u} - u \ge -f$  on $\ol{\Om}$, we then infer that
\begin{equation*}
   \int_\Ga h_\nu^2\, dS_x \le c\int_\Om (\ol{u}-u)\,|\na^2 h|^2\,dx.
\end{equation*}
Hence, this formula and the trivial inequalities 
\begin{equation*}
\nr \ol{u} - u \nr_{\infty,\Ga} \leq \osc_{\Ga} u , \, \
\nr \ol{u} - u \nr_{2,\Ga} \le \nr \ol{u} - u \nr_{1,2,\tau,\Ga}, \, \
\nr \na_\Ga u \nr_{2,\Ga} \leq \nr \ol{u} - u \nr_{1,2,\tau,\Ga}
\end{equation*} 
give that 
\begin{multline*}
\int_\Om (\ol{u}-u)\,|\na^2 h|^2\,dx \le \\
c\, \nr \ol{u} - u \nr_{1,2,\tau,\Ga}
    \left[ \Bigl(1+\osc_{\Ga} u \Bigr) \nr \ol{u} - u \nr_{1,2,\tau,\Ga} +\sqrt{\int_\Om (\ol{u}-u)\,|\na^2 h|^2\,dx}   \right].
\end{multline*}
This inequality easily gives the desired bound.
\end{proof}

\smallskip

We are now in position to prove our quantitative estimates for the radial symmetry in the reverse Serrin problem.

\begin{proof}[Proof of Theorem \ref{thm:prova}]
Set 
\begin{equation}\label{eq:definition sigma in proof main thm}
\displaystyle \si = \min \left\{1, \frac{r_i^2}{4}\right\}. 
\end{equation}
We first notice that if  
$$
    \osc_\Ga u + \nr \na_\Ga u \nr_{\infty, \Ga} \ge \si,
$$
then the conclusion trivially holds true. In fact, by recalling the definition of $\psi$ in \eqref{eq:old-profile-psi}, we easily get that  
$$
    \rho_e - \rho_i \le d_\Om \le 
    \begin{cases}
           \frac{d_\Om}{\si} \, \psi\left( \osc_\Ga u + \nr \na_\Ga u \nr_{\infty, \Ga} \right)
            & \text{if } N=2,3 ,
            \\
            \frac{d_\Om}{\si^{2/(N-1)}} \, \psi \left( \osc_\Ga u + \nr \na_\Ga u \nr_{\infty, \Ga} \right)
            & \text{if } N \geq 4.
    \end{cases}
$$ 
Thus, it only remains to prove the result when
\begin{equation*}
	\osc_\Ga u + \nr \na_\Ga u \nr_{\infty, \Ga} < \si.
\end{equation*}
\par
Under this assumption and recalling \eqref{torsion} and \eqref{neumann}, it is clear that
	\begin{equation}\label{eq:free gradient bound}
	\nr \na u \nr_{\infty, \Om} = \nr \na u \nr_{\infty, \Ga} = \nr \na_\Ga u + u_\nu \,  \nu \nr_{\infty, \Ga} < \si +R 
	\le \si + d_\Om.
	\end{equation}
The last inequality follows  from \eqref{def-R}
and  by putting together the isoperimetric inequality $|\Ga| \ge N |B|^{1/N} |\Om|^{(N-1)/N}$ and the trivial bound $|\Om| \le |B| \, d_\Om^N$.

	As usual, we consider the harmonic function $h=q-u$ and choose the minimum point of $q$ to coincide with the point $z$ in our assumptions. Since $z$ is a minimum point for both $q$ and $u$,
we have that $\na h(z)=0$. 

Thanks to Lemma \ref{lemma geo bound -u > de^2}, the left-hand side of identity \eqref{neumann-identity-intro} can be estimated from below as
    \begin{equation}\label{ineq: de |na^2 h|^2 < (ol{u}-u) |na^2 h|^2}
    \frac12\,r_i\int_\Om \de_\Ga \, |\na^2 h|^2 \,dx \le \int_\Om (\ol{u}-u)\, |\na^2 h|^2 \,dx .
    \end{equation}
    Now, from \eqref{bound-h-rho} we infer that 
    \begin{equation}\label{ineq: (ro_e - ro_i) < osc h + osc u}
        \frac12 \left(\frac{|\Om|}{|B|}\right)^\frac1{N}(\rho_e-\rho_i)\le
        \osc_{\Ga} q \leq \osc_{\Ga} h + \osc_{\Ga} u.
    \end{equation}
    Combining Lemma \ref{thm osc bound}, Lemma \ref{lem bound in W_1,2,Ga}, \eqref{ineq: de |na^2 h|^2 < (ol{u}-u) |na^2 h|^2}, and the trivial inequality $\sqrt{a^2+b^2} \leq |a|+|b|$, 
we infer that
    \begin{equation}\label{eq:bound osc h in proof main thm}
        \osc_\Ga h \leq c \, \psi\left( \osc_\Ga u + \nr \na_\Ga u \nr_{\infty, \Ga} \right),
    \end{equation}
    for some explicit constant $c$ that only depends on $N, d_\Om$, $r_e$, and $r_i$. 
    Here, we used two facts. Firstly, \eqref{eq:free gradient bound}, \eqref{eq:definition sigma in proof main thm} and the trivial bound $\nr \na q \nr_{\infty,\Om} \le d_\Om$ give the explicit upper bound for $|\na h|$:
    $$
    \nr \na h \nr_{\infty, \Om} \le \nr \na q \nr_{\infty, \Om} + \nr \na u \nr_{\infty, \Om} \le \min \left\{1, \frac{r_i^2}{4}\right\} + 2 d_\Om.
    $$
Secondly, combining Remark \ref{rem:possibilestima de(z)},
\eqref{eq:free gradient bound}, \eqref{eq:definition sigma in proof main thm}, and the current smallness assumption on the deviation of $u$ from being constant on $\Ga$, gives the explicit lower bound for $\de_{\Ga}(z)$:
    $$
    \de_{\Ga}(z) \ge \frac{r_i^2}{2\left( \min\left\lbrace 1, \frac{r_i^2}{4} \right\rbrace + d_\Om \right)}.
    $$
    
    Finally, by the trivial inequality $\osc_\Ga u \leq \osc_\Ga u + \nr \na_\Ga u \nr_{\infty, \Ga}$ and the definition of $\psi$, we easily obtain that
    \begin{equation*}
        \psi\left( \osc_\Ga u + \nr \na_\Ga u \nr_{\infty, \Ga} \right) + \osc_\Ga u \leq C \psi\left( \osc_\Ga u + \nr \na_\Ga u \nr_{\infty, \Ga} \right),
    \end{equation*}
    with
    \begin{equation*}
    	C=\begin{cases}
    		2
    		& \text{if } N=2,3,
    		\\
    		1+ \si^{(N-3)/(N-1)}  
    		& \text{if } N \geq 4 ,
    	\end{cases}
    \end{equation*}
    where $\si$ is that defined in \eqref{eq:definition sigma in proof main thm}.

   The conclusion easily follows by putting together the last inequality, \eqref{eq:bound osc h in proof main thm}, and \eqref{ineq: (ro_e - ro_i) < osc h + osc u}.
\end{proof}

\begin{proof}[Proof of Theorem \ref{thm:stability-reverse-serrin}]
We shall adapt the argument used in the proof of Theorem \ref{thm:prova} to the setting with the current deviation.
\par
Let $\si$ be given by \eqref{eq:definition sigma in proof main thm}. As before, it is clear that it is enough to check the case in which
\begin{equation*}
    \osc_\Ga u \, + \|\na_\Ga u\|_{2,\Ga} < \si.
\end{equation*}

This time, we observe that
\begin{equation}\label{eq:free gradient bound 2}
	\nr \na u \nr_{\infty, \Om} \le \nr \na q \nr_{\infty, \Om} + \nr \na h \nr_{\infty, \Om}  \leq d_{\Om} + \nr \na h \nr_{\infty,\Om}.
\end{equation}
An explicit bound for $\nr \na h \nr_{\infty, \Om}$ will be obtained at the end of the proof.
\par
Also, as before we can use 
\eqref{ineq: de |na^2 h|^2 < (ol{u}-u) |na^2 h|^2}. 
Similarly, by proceeding as in the proof of Theorem \ref{thm:prova}, with 
$\osc_\Ga u + \nr \na_\Ga u \nr_{\infty, \Ga}$ replaced by $\osc_\Ga u + \nr \na_\Ga u \nr_{2, \Ga}$, we arrive at the inequalities
\begin{equation}
\label{eq:inequality trivial for proof main thm 2}
	\psi\left( \osc_\Ga u + \nr \na_\Ga u \nr_{2, \Ga} \right) + \osc_\Ga u \leq C\, \psi\left( \osc_\Ga u + \nr \na_\Ga u \nr_{2, \Ga} \right)
\end{equation}
and
\begin{equation}\label{eq:osc h in main thm 2}
        \osc_\Ga h \leq c \, \psi\left( \osc_\Ga u + \nr \na_\Ga u \nr_{2, \Ga} \right).
    \end{equation}
Here, $\psi$ is that in \eqref{eq:old-profile-psi}, $C$ is the same constant appearing in the proof of Theorem \ref{thm:prova}, and $c$ is an explicit constant that only depends on $N, d_\Om$, $r_i$, $r_e$, a lower bound for $\de_{\Ga}(z)$, and an upper bound for $\nr \na h \nr_{\infty,\Om}$. 
\par
The desired conclusion will easily follow from \eqref{ineq: (ro_e - ro_i) < osc h + osc u}, \eqref{eq:osc h in main thm 2}, and \eqref{eq:inequality trivial for proof main thm 2}, after obtaining suitable bounds for $\de_\Ga (z)$ and $ \nr \na h \nr_{\infty,\Om}$.

By putting together \eqref{eq:definition sigma in proof main thm}, Remark \ref{rem:possibilestima de(z)}, \eqref{eq:free gradient bound 2}, and the current smallness assumption on the deviation of $u$ from being constant on $\Ga$,  we easily get that
$$
\de_\Ga (z) \ge \frac{r_i^2}{4( d_\Om + \nr \na h \nr_{\infty,\Om} )}.
$$
We are thus left to obtain an upper bound for $\nr \na h \nr_{\infty,\Om}$.
\par
Thanks to Theorem~\ref{th:bound-derivatives} (with $m=1$) and recalling \cite[Corollary 3.14]{ABMMZ}, $\nr\na h\nr_{\infty,\Om}$ can be clearly bounded by a constant that 
depends on $N, d_\Om, r_i$, and $r_e$.

When $\Ga$ is of class $C^{2,\al}$, we can obtain an improvement of the profile $\psi$, by arguing  as in \cite[Lemma 4.2 and Corollary~4.3]{MP5} (with $p=2$, $q=+\infty$). In fact, we infer that
$$
\nr \na h \nr_{\infty,\Om} \le c\, \nr \na^2 h \nr_{\infty,\Om}^{(N-1)/(N+1)} \nr \de_\Ga^{1/2} \na^2 h \nr_{2,\Om}^{2/(N+1)},
$$
for some $c$ depending on $N, d_\Om$, $r_i$, $r_e$. Hence, by recalling Theorem \ref{th:bound-derivatives} (with $m=2$), we easily get that
$$
\nr \na h \nr_{\infty,\Om} \le c \, \nr \de_\Ga^{1/2} \na^2 h \nr_{2,\Om}^{2/(N+1)},
$$
where $c$ depends on $N, d_\Om,$ and the $C^{2,\al}$-regularity of $\Ga$.
\par
Plugging the last inequality in Lemma \ref{thm osc bound} and using Lemma \ref{lem bound in W_1,2,Ga} and \eqref{ineq: de |na^2 h|^2 < (ol{u}-u) |na^2 h|^2} as before, gives \eqref{eq:osc h in main thm 2} with the improved profile shown in \eqref{eq:improvementMinE}. Clearly, for this profile, \eqref{eq:inequality trivial for proof main thm 2} remains true with $C$ given by
\begin{equation*}
	C=\begin{cases}
		2
		& \text{if } N=2,3,
		\\
		1+ \si^{(N-3)/(N+1)}  
		& \text{if } N \geq 4.
	\end{cases}
\end{equation*}

Thus, combining \eqref{ineq: (ro_e - ro_i) < osc h + osc u}, \eqref{eq:osc h in main thm 2}, and \eqref{eq:inequality trivial for proof main thm 2} gives the improved stability for $N \ge 4$.
\end{proof}

\section{Stability via the classical Serrin problem}\label{sec:discussion reverse Serrin via Serrin}
\label{sec:trick}
We now present a quantitative bound for the reverse Serrin problem that can be obtained exploiting the existing stability results for the classical Serrin problem. The proof is based on a simple trick. As discussed in the Introduction, we get a poorer estimate than those obtained in Theorems \ref{thm:prova} and \ref{thm:stability-reverse-serrin}. 
In fact, the main drawback is that the deviations $\osc_\Ga u+\|\na_\Ga u\|_{\infty,\Ga}$ and $\osc_\Ga u+\|\na_\Ga u\|_{2,\Ga}$ used in Theorems \ref{thm:prova} and \ref{thm:stability-reverse-serrin} must be replaced by the much more stringent deviation $\| \ol{u}-u \|_{C^{1, \al }(\Ga)}$.

We start with an estimate for an auxiliary function.

\begin{prop}
		\label{lem:for alternative stability-reverse-serrin}
		Let $\Om \subset \RR^N$, $N \geq 2$, be a bounded domain with boundary $\Ga$ of class $C^{1,\al}$, $0<\al<1$. 
		Let $u$ be a solution of \eqref{torsion}, \eqref{neumann}. 
		\par
Set $f=u-w$, where $w$ is the solution of the Dirichlet problem
		\begin{equation*}
			\De w=0 \ \mbox{ in } \ \Om, 
			\qquad
			w = u \ \text{ on } \ \Ga.
		\end{equation*}
Then, we have that 
\begin{equation*}
			\nr R - f_\nu \nr_{C^{0,\al}(\Ga)} \le c \, \| \ol{u}-u \|_{C^{1, \al }(\Ga)}.
		\end{equation*}
Here, $c$ only depends on $N$ and the $C^{1, \al}$-regularity of $\Ga$. 
	\end{prop}
	\begin{proof}
We notice that
	\begin{equation*}
	\nr	R - f_\nu \nr_{C^{0,\al}(\Ga)} = \nr w_\nu \nr_{C^{0,\al}(\Ga)} \le \nr \ol{u}-w \nr_{C^{1,\al}(\ol{\Om})}.
	\end{equation*}
The desired conclusion then follows from \cite[Theorem 8.33]{GT}.
	\end{proof}
	
	\begin{thm}[The reverse Serrin problem with strong deviation]
		\label{thm:Alternative-stability-reverse-serrin}
		Let $\Om $ be a bounded domain in $\RR^N$, $N \geq 2$,
		with boundary $\Ga$ of class $C^{1,1}$. 
\par
Let $u$ be solution of \eqref{torsion}, \eqref{neumann}, 
and let $z$ be any global minimum point of $u$ on $\ol{\Om}$. 
Then, we have that 
		\begin{equation*}
			\rho_e-\rho_i \le c \, \psi\left( \| \ol{u}-u \|_{C^{1, \al }(\Ga)} \right),
		\end{equation*}
		where $\psi$ is the profile defined in \eqref{eq:old-profile-psi}. The constant $c$ only depends on $N$, $d_\Om$, $r_i$, and $r_e$.
\par
For $N \ge 4$ the profile $\psi$ can be improved to obtain \eqref{eq:improvementMinE}, at the cost of replacing the dependence of $c$ on the $r_i$ and $r_e$ with that on the $C^{2,\al}$-regularity of $\Ga$.  
 
	\end{thm}
	\begin{proof}
It is clear that the function $f$ defined in Proposition \ref{lem:for alternative stability-reverse-serrin} satisfies:
$$
\De f=N \ \mbox{ in } \ \Om, \ f=0 \ \mbox{ on } \ \Ga.
$$
By \cite[Theorem 3.1]{MP3} (for $N\neq3$) and \cite[Theorem 4.4]{MP5} (for $N=3$), we have that
		$$
		\rho_e -\rho_i \le c \, \psi\left( \nr f_\nu - R \nr_{L^2 (\Ga)} \right) , 
		$$
where $c$ only depends on $N$, $d_\Om$, $r_i$, and $r_e$.
Here, $\psi$ is as in \eqref{eq:old-profile-psi}. 
\par
Furthermore, the improvement of the profile $\psi$ to \eqref{eq:improvementMinE} for $N \ge 4$ is obtained from
\cite[Theorem 4.4]{MP5}, at the aforementioned cost.
\par	
The conclusion then follows from Proposition  \ref{lem:for alternative stability-reverse-serrin} and by noting that
$$
		\nr f_\nu - R \nr_{L^2 (\Ga)} \le |\Ga|^{1/2} \nr f_\nu - R  \nr_{L^\infty(\Ga)} \le  |\Ga|^{1/2} \nr f_\nu - R \nr_{ C^{0,\al} (\Ga)}.
$$
As usual, $|\Ga|$ can be estimated in terms of the desired parameters by recalling \cite[Remark 3.1]{MP5}.
\end{proof}

\begin{rem}
{\rm	
The argument may be adjusted to work also in the semilinear case, at the cost of replacing the application of \cite[Theorem 3.1]{MP3} and \cite[Theorem 4.4]{MP5} by the results in \cite{ABR,CMV}, and hence obtaining a more general but poorer stability estimate.
}
\end{rem}


\section*{Acknowledgements}
The authors wish to thank S. Dipierro, M. Onodera, and E. Valdinoci for useful discussions on Section \ref{sec:trick}.
\par
R. Magnanini is partially supported by the PRIN grant  n. 201758MTR2 of the Ministero dell'Universit\`a e della Ricerca (MUR).

R. Molinarolo is partially supported by the SPIN Project  ``DOMain perturbation problems and INteractions Of scales - DOMINO''  of the Ca' Foscari University of Venice.

G. Poggesi is supported by the Australian Research Council (ARC) Discovery Early Career Researcher Award (DECRA) DE230100954 ``Partial Differential Equations: geometric aspects and applications'' and the 2023 J~G~Russell Award from the Australian Academy of Science. He is member of the Australian Mathematical Society (AustMS).

The authors are also partially supported by the Gruppo Nazionale Analisi Mate\-matica Probabilit\`a e Applicazioni (GNAMPA) of the Istituto Nazionale di Alta Matema\-ti\-ca (INdAM).

\end{document}